\documentclass{conm-p-l}

\usepackage{amscd}
\usepackage{amssymb}
\usepackage[all]{xy}

\numberwithin{equation}{section}

\theoremstyle{plain} 
\newtheorem{theorem}[equation]{Theorem}
\newtheorem{lemma}[equation]{Lemma}

\newtheorem{proposition}[equation]{Proposition}
\newtheorem{corollary}[equation]{Corollary}

\newtheorem*{maintheorem}{Theorem~\ref{thm: main theorem}}

\theoremstyle{definition}
\newtheorem{definition}[equation]{Definition}

\newtheorem*{remark*}{Remark}

\newcommand{\maintheoremtext}{
Suppose that $H$ is a $p$-toral subgroup of $U(n)$
and $\weakfixed{H}$ is not contractible. Then $H$ is a projective
elementary abelian $p$-subgroup of $U(n)$.
}

\newcommand{\reals}{{\mathbb{R}}}
\newcommand{\integers}{{\mathbb{Z}}}

\newcommand{\complexes}{{\mathbb{C}}}

\newcommand{\lambdaglom}[1]{ (\lambda/{ #1 })}

\newcommand{\n}{\mathbf{n}} 

\def\doCal#1{%
\ifx#1\doAllCalEnd\def\doAllCal{\relax}\else%
 \expandafter\edef\csname#1cal\endcsname{{\noexpand\mathcal #1}}\fi}
\def\doAllCal#1{\doCal#1\doAllCal}
\doAllCal ABCDEFGHIJHLMNOPQRSTUVWXYZ\doAllCalEnd

\def\doBar#1{%
\ifx#1\doAllBarEnd\def\doAllBar{\relax}\else%
 \expandafter\edef\csname#1bar\endcsname{{\noexpand\overline{#1}}}\fi}
\def\doAllBar#1{\doBar#1\doAllBar}
\doAllBar ABCDEFGHIJHLMNOPQRSTUVWXYZabcdefghijhlmnopqrstuvwxyz\doAllBarEnd
\def\doWiggle#1{%
\ifx#1\doAllWiggleEnd\def\doAllWiggle{\relax}\else%
 \expandafter\edef\csname#1wiggle\endcsname{{\noexpand\tilde{#1}}}\fi}
\def\doAllWiggle#1{\doWiggle#1\doAllWiggle}
\doAllWiggle ABCDEFGHIJHLMNOPQRSTUVWXYZabcdefghijklmnopqrstuvwxyz\doAllWiggleEnd

\DeclareMathOperator{\class}{cl}

\DeclareMathOperator{\GL}{GL}
\DeclareMathOperator{\Gr}{Gr}

\DeclareMathOperator{\Id}{Id}

\DeclareMathOperator{\Morph}{Morph}
\DeclareMathOperator{\Nerve}{Nerve}

\DeclareMathOperator{\Obj}{Obj}
\DeclareMathOperator{\rank}{rank}

\DeclareMathOperator{\st}{st}

\newcommand{\Iso}[1]{\left(\Lcal_{n}\right)^{#1}_{\operatorname{iso}}}
\newcommand{\Isospecific}[2]{\left(\Lcal_{#2}\right)^{#1}_{\operatorname{iso}}}
\newcommand{\orderp}[1]{#1[p]}
\newcommand{\weakfixed}[1]{ \left(\Lcal_{n}\right)^{#1} }
\newcommand{\weakfixedspecific}[2]{ \left(\Lcal_{#2}\right)^{#1} }
\newcommand{\strongfixed}[1]{ \left(\Lcal_{n}\right)^{#1}_{\st}}
\newcommand{\strongfixedspecific}[2]{ \left(\Lcal_{#2}\right)^{#1}_{\st}}
\newcommand{\weakPfixed}[1]{ \left(\Pcal_{n}\right)^{#1} }

\begin{document}

\title[$p$-toral fixed point sets on $\Lcal_{n}$]
{Fixed points of $p$-toral groups acting on partition complexes}

\author{Julia E.\ Bergner}
\address{Department of Mathematics, University of
  California, Riverside}
\email{bergnerj@member.ams.org}
\thanks{The first, third, and fourth authors received partial
  support from NSF grants DMS-1105766, DMS-0968251, and DMS-1307390, respectively.  The second author was partially supported by DFG grant HO 4729/1-1, and the fifth author was partially supported by an AIM $5$-year fellowship.}
\author{Ruth Joachimi}
\address{Department of Mathematics and Informatics, University of Wuppertal, Germany}
\email{joachimi@math.uni-wuppertal.de}
\author{Kathryn Lesh}
\address{Department of Mathematics, Union College, Schenectady, NY 12309 USA}
\email{leshk@union.edu}
\author{Vesna Stojanoska}
\address{Department of Mathematics, Massachusetts Institute of Technology, Cambridge MA 02143}
\email{vstojanoska@math.mit.edu}
\author{Kirsten Wickelgren}
\address{Georgia Institute of Technology, School of Mathematics, 686 Cherry Street, Atlanta GA 30332}
\email{wickelgren@post.harvard.edu}


\maketitle
\markboth{\sc{Bergner, Joachimi, Lesh, Stojanoska, and Wickelgren}}
{\sc{$p$-toral fixed point sets in $\Lcal_{n}$}}

\begin{abstract}
  We consider the action of $p$-toral subgroups of $U(n)$ on the
  unitary partition complex $\Lcal_{n}$. We show that if $H\subseteq
  U(n)$ is $p$-toral and has noncontractible fixed points on
  $\Lcal_{n}$, then the image of $H$ in the projective unitary group
  $U(n)/S^{1}$ is an elementary abelian $p$-group.
\end{abstract}

\maketitle

\section{Introduction}
Let $\n$ denote the set $\{1,...,n\}$ and let $\Pcal_{n}$ denote the
nerve of the poset of proper, nontrivial partitions of $\n$, ordered
by coarsening. In~\cite{ADL2}, Arone, Dwyer, and Lesh compute the
Bredon homology of $\Pcal_{n}$ for certain kinds of $p$-local Mackey
functors on the category of $\Sigma_{n}$-sets.  The calculation is
part of a program to obtain a proof of the Whitehead Conjecture and
the collapse of the homotopy spectral sequence of the Goodwillie tower
of the identity functor for $S^{1}$ by using the Bousfield-Kan
cosimplicial resolution of $S^{1}$.  A key element in the calculation
of~\cite{ADL2} is understanding which $p$-subgroups $H \subseteq
\Sigma_{n}$ can have noncontractible fixed point sets
$\weakPfixed{H}$. It turns out that if $H \subseteq \Sigma_{n}$ is a
$p$-group and $\weakPfixed{H}$ is not contractible, then $H$ is
elementary abelian (Proposition~6.6 in~\cite{ADL2}).

In this paper, we consider the corresponding question in the unitary
context, following analogies set up by Arone and Lesh
in~\cite{Arone-Topology} and~\cite{Arone-Lesh}.  Let $\Lcal_{n}$
denote the nerve of the (topological) poset of proper partitions of
$\complexes^{n}$ into orthogonal subspaces, where ``proper'' means
that we exclude the partition consisting of the single subspace
$\complexes^{n}$ itself. (See Section~\ref{section: Ln small}.)  The
action of the unitary group $U(n)$ on $\complexes^n$ induces an action
on $\Lcal_n$.  The space $\Lcal_{n}$ with its $U(n)$ action is
strongly related to the $bu$-analogues of symmetric powers of spheres
and to the Weiss tower for the functor $V\mapsto BU(V)$ (see
\cite{Arone-Lesh}~Theorem~9.5 and \cite{Arone-Topology} Theorems~2
and~3).

In moving from finite group theory to compact Lie groups, one replaces the
notion of a $p$-group with that of a $p$-toral group, i.e., an
extension of a finite $p$-group by a torus.  In this paper we study
the action of $p$-toral subgroups $H\subseteq U(n)$ on $\Lcal_n$ in
order to find out when the fixed point set $\weakfixed{H}$ is
contractible.  We show that the answer is ``most of the time": the
condition that $\weakfixed{H}$ is not contractible puts considerable
group-theoretic restrictions on the $p$-toral subgroup $H$.

Recall that the center of $U(n)$ is $S^1$, and that the projective unitary
group is defined as $PU(n)=U(n)/S^{1}$. A subgroup $H$ of $U(n)$ is called
\emph{projective elementary abelian} if its image in $PU(n)$ is
elementary abelian. Our main theorem is the following.
 
\begin{theorem}     \label{thm: main theorem}
\maintheoremtext
\end{theorem}

The organization of the paper is as follows. 
In Section~\ref{section: Ln small}, we define $\Lcal_{n}$ and
investigate the first nontrivial examples ($n=2,3$) directly and in
detail.  Section~\ref{section:Group} does some group-theoretic setup
that is needed for studying fixed points.  In
Section~\ref{section:Conditions}, we find a condition on~$H$ that
implies contractibility of $\weakfixed{H}$
(Theorem~\ref{thm:ContractJ}), and in Section~\ref{section:MainProof}
we prove Theorem~\ref{thm: main theorem}.

Finally, in Section~\ref{section: examples of fixed points} we compute two
illustrative examples of fixed points.  First, we exhibit a projective
elementary abelian $2$-subgroup of $U(2)$ whose fixed points on
$\Lcal_{2}$ are not contractible (Proposition~\ref{prop:
  non-contractible}), which shows that Theorem~\ref{thm: main theorem}
is group-theoretically sharp. In contrast, we also show that the same
subgroup, embedded in $U(3)$, has contractible fixed points on
$\Lcal_{3}$ (Proposition~\ref{proposition: L3 fixed contractible}).
In future work \cite{Banff-follow-on}, we plan to establish further
restrictions on the projective abelian $p$-subgroups of $U(n)$ that
can have noncontractible fixed point sets by considering their
representation
theory in greater detail.\\

\noindent{\bf{Notation and Terminology}}

We generally do not distinguish notation for a category and its
nerve, and we trust to context to make clear which is being discussed.

By ``subgroup,'' we always mean a closed subgroup. If $G$ is a group,
we write $Z(G)$ for the center of $G$ and $\orderp{G}$ for the elements 
of $G$ of order $p$. \\

\noindent{\bf{Acknowledgements:}}
The authors thank the Banff International Research Station and the 
Clay Mathematics Institute for financial support, and the anonymous referee 
for a helpful and thorough reading of the paper. The first and third
authors are grateful to Bill Dwyer for preliminary discussions about this 
project. 

\section{Homotopy type of $\Lcal_{n}$ for small $n$}
\label{section: Ln small}

In this section, we give an introduction to $\Lcal_{n}$. 
We define and describe it in some detail, and we look at the two 
lowest-dimensional examples. 

The partition complex $\Lcal_{n}$ is a poset category internal to
topological spaces.  An object $\lambda$ in $\Lcal_{n}$ is a proper,
unordered decomposition of $\complexes^n$ into nonzero, mutually orthogonal
subspaces $v_1, \ldots, v_m$, where by \emph{proper} we mean that
$m>1$.  To topologize the set of objects,
$\Obj\left(\Lcal_{n}\right)$, let $\Gr_{k}\left(\complexes^{n}\right)$
denote the Grassmannian of $k$-planes in $\complexes^{n}$.  The
set of objects of $\Lcal_{n}$ is given the subspace topology
\[
\Obj\left(\Lcal_{n}\right)
\subseteq 
\coprod_{m>1}
\left[
\left(
         \coprod_{k\geq 1} 
       \Gr_{k}\left(\complexes^{n}\right)
\right)^{\!m}
/\Sigma_{m}\right].
\]
Note that the connected components $\Obj\left(\Lcal_{n}\right)$ are in
one-to-one correspondence with unordered partitions of the integer $n$
as the sum of at least two positive integers.

Morphisms in $\Lcal_{n}$ are given by coarsenings; that is, there is a
morphism from $\{v_1, \ldots, v_m\}$ to $\{w_1, \ldots, w_{m'}\}$ if
and only if for each $v_i$ there exists a $j$ such that
$v_{i}\subseteq w_{j}$.  Note that between any two objects there is at
most one morphism.
In particular, there is a monomorphism
\[
\Morph\left(\Lcal_{n}\right)
\hookrightarrow \Obj\left(\Lcal_{n}\right)\times \Obj\left(\Lcal_{n}\right)
\]
via the source and target maps, and accordingly we topologize the morphism
space by the subspace topology of the product topology. 

Since $\Lcal_{n}$ is a category internal to topological spaces, its
nerve is a simplicial space (or the realization of that simplicial
space, depending on context). Simplicial degree zero of the nerve
consists simply of the space of objects of $\Lcal_{n}$. The first
simplicial degree consists of the space of morphisms, the second
simplicial degree consists of the space of composable morphisms
(topologized as a subspace of the two-fold product of the morphism
space), and so forth. (See, for example, Section 5.1 of~\cite{Libman}.)

The action of $U(n)$ on $\complexes^{n}$ induces a continuous action
of $U(n)$ on the category $\Lcal_{n}$ (i.e., the action is continuous
on the space of objects and the space of morphisms). Hence the nerve
of $\Lcal_{n}$ likewise has an action of $U(n)$, as do the nerves of
any subcategories closed under the action of $U(n)$. 
By inspection of the simplices in the nerve of $\Lcal_{n}$, we see that 
\[
\left(\Nerve\left(\Lcal_{n}\right)\right)^{H}
        \cong \Nerve\left(\weakfixed{H}\right).
\]

To provide the reader with some intuition about $\Lcal_{n}$, at least
in low dimensions, we work out concrete information about $\Lcal_{2}$
and $\Lcal_{3}$.  (Observe that $\Lcal_{1}$ is empty, since
$\complexes$ has no proper partitions.)  For $\Lcal_{2}$, the smallest
interesting example of a unitary partition complex, we can actually
find the homeomorphism type.  For the more complicated example of
$\Lcal_3$, we exhibit its homotopy type as a homotopy pushout diagram
and prove that it is simply connected (Propositions~\ref{prop: pushout
  diagram for L3} and~\ref{proposition: simply connected}).

To begin our study of $\Lcal_{2}$, we observe that a proper partition
of $\complexes^{2}$ can only be a partition into two orthogonal lines.
Since there are no refinements and no proper coarsenings of such a
partition, the poset category of partitions of $\complexes^{2}$ has only
identity morphisms, and $\Lcal_{2}$ is homeomorphic to its space of
objects.

\begin{proposition}   \label{prop: RP2}
The space $\Lcal_2$ is homeomorphic to $\reals P^2$.
\end{proposition}

\begin{proof}
A partition of $\complexes^{2}$ is an unordered pair consisting of a
  line in $\complexes^2$ and its orthogonal complement. The space of
  lines in $\complexes^{2}$ is the projective space $\complexes P^1$.
  Because the pair is unordered, $\Lcal_2$ is the quotient of
  $\complexes P^1$ by the action of the involution that
  interchanges a line and its orthogonal complement.

More explicitly, note that the line spanned by
  $(0,1)$ has orthogonal complement spanned by $(1,0)$ (a special
  case), and in general the line in $\complexes^2$ spanned by
  $(1,z)$ with $z\in\complexes\backslash\{0\}$ has orthogonal
  complement spanned by $(1, -1/\zbar)$.  Thus $\Lcal_{2}$ is
  homeomorphic to the quotient of $S^{2}\cong\complexes P^{1}\cong
  \complexes\cup\{\infty\}$ by the involution $z \mapsto -1/\zbar$
  (and $0\leftrightarrow{}\infty$).  The involution exchanges
  points in the region $\|z\|>1$ with those in the region $\|z\|<1$,
  so we only need to consider the quotient of the unit disk $\|z\|\leq
  1$ by the action on the boundary circle. When $\|z\| =1$, we can
  write $z= e^{i \varphi}$ and $-1/\zbar= -e^{i\varphi}$ for some
  $\varphi\in\reals$, whence the transformation is the antipodal map
  on the boundary of the unit disk.  We conclude that $\Lcal_{2}$ is
  homeomorphic to the quotient space obtained from the disk $\|z\|\leq
  1$ by identifying antipodal points on the boundary circle, namely
  $\reals P^{2}$.
\end{proof}

We now turn to $\Lcal_3$, which is more complicated.  There are two
connected components in the space of objects, corresponding to
partitions of $\complexes^{3}$ into three lines and partitions into a
line and a $2$-plane. The action of $U(n)$ on $\Lcal_{n}$ allows the
following explicit identification of the connected components of the
object space: the transitive action of $U(n)$ on each connected
component exhibits each component as the homogeneous space $U(n)/I$
once we have computed the isotropy group $I$ of a typical object in
the component.  A decomposition of $\complexes^3$ into a
$1$-dimensional subspace $v_{1}$ and its 2-dimensional orthogonal
complement $v_{2}$ has isotropy group conjugate to $U(1)\times U(2)$,
since an element of $U(3)$ that stabilizes the partition $\{v_{1},
v_{2}\}$ must stabilize $v_{1}$ and $v_{2}$ individually.  On the
other hand, elements of the isotropy group of a decomposition of
$\complexes^3$ into three lines can act nontrivially on each line, but
can also permute the lines, because the lines all have the same
dimension. Hence this isotropy group is
$\left(U(1)\right)^3\rtimes\Sigma_3$. We conclude that the object
space has homeomorphism type
\[
\left[\strut U(3)/\left(U(1)\times U(2)\right)\right]
\ \sqcup\  
\left[\strut U(3)/\left( \left(U(1)\right)^3\rtimes\Sigma_3 \right)\right].
\]
We write $\Gr(1,2)$ for the first component, and $\Gr(1,1,1)$ for the second. 

The next task is to identify the morphism space of $\Lcal_{3}$. Each
connected component of the object space gives a connected component of
the morphism space consisting of identity morphisms, so two components
of the morphism space of $\Lcal_{3}$ are given by $\Gr(1,2)$ and
$\Gr(1,1,1)$. (These components are precisely the degenerate simplices
in simplicial dimension $1$ of the nerve of $\Lcal_{3}$.)

Unlike $\Lcal_{2}$, the category $\Lcal_{3}$ has nonidentity
morphisms, given by coarsenings from $\Gr(1,1,1)$ to $\Gr(1,2)$. The
action of $U(3)$ on these morphisms is transitive, so again we can
identify the homeomorphism type of this component of the morphism
space by finding the isotropy group of a sample morphism, say the
morphism
\[
\complexes\oplus\complexes\oplus\complexes 
   \longrightarrow \complexes\oplus\complexes^{2}
\]
that takes the standard basis to itself in the natural way. 
There is exactly one morphism between these objects, 
so for it to be fixed
it is necessary and sufficient that both the source and the target
be fixed. The isotropy group $I$ of the morphism is therefore the intersection
of the isotropy groups of the source and the target, that is, 
\begin{align*}
I &= \left(U(1)^3\rtimes\Sigma_3 \right)
     \cap \left(\strut U(1)\times U(2)\right) \\
 & = U(1)\times \left(\strut U(1)^{2}\rtimes\Sigma_{2} \right).
\end{align*}

\begin{proposition}   \label{prop: pushout diagram for L3}
The nerve of $\Lcal_{3}$ is homeomorphic to the double mapping 
cylinder of the diagram
\begin{equation}\label{eq: pushout diagram for L3}
\begin{CD}
U(3)/I @>{\qquad\qquad }>>U(3)/\left(\strut U(1)\times U(2)\right)\\
@VVV \\
U(3)/\left( \left(U(1)\right)^3\rtimes\Sigma_3 \right). 
\end{CD}
\end{equation}
\end{proposition}

\begin{proof}
The nerve of the category $\Lcal_{3}$ has nondegenerate simplices only 
in simplicial dimensions $0$ and $1$, since there are no composable morphisms
that do not involve an identity morphism. 
Diagram~\eqref{eq: pushout diagram for L3} has the two connected components
of the object space of $\Lcal_{3}$ in the upper right and lower left corners,
and the nondegenerate part of the morphism space in the upper left corner. 
The double mapping cylinder is homeomorphic to the realization of the 
simplicial space that gives the nerve of $\Lcal_{3}$. 
\end{proof}

From Proposition~\ref{prop: pushout diagram for L3},
we obtain the following homotopy-theoretic result. 
\begin{proposition}   \label{proposition: simply connected}
The space $\Lcal_3$ is simply connected.
\end{proposition}

\begin{proof}
  We use diagram~\eqref{eq: pushout diagram for L3} and the
  Seifert-Van Kampen theorem, noting that all of the spaces in
  \eqref{eq: pushout diagram for L3} are path connected. To find the
  fundamental groups of the corners, recall that if $G$ is a 
  compact Lie group with maximal torus~$T$, then the natural map
  $\pi_{1}T\rightarrow\pi_{1}G$ is an epimorphism (e.g., Corollary~5.17
  in \cite{MimuraToda}).
  For the fundamental group of $U(3)/(U(1)\times U(2))$, we observe
  that $U(1)\times U(2)$ contains the maximal torus of $U(3)$ and so
  $U(1)\times U(2)\rightarrow U(3)$ induces a surjection on
  fundamental groups. Since $U(1)\times U(2)$ is connected, we
  conclude from the fiber sequence
\[ 
U(1) \times U(2) \rightarrow U(3) \rightarrow U(3)/(U(1) \times U(2))
\] 
that the upper right corner of
\eqref{eq: pushout diagram for L3} is simply connected.

To find the fundamental group of the lower left corner 
of~\eqref{eq: pushout diagram for L3}, we consider the fiber sequence 
\[ 
U(1)^3\rtimes\Sigma_3 
         \rightarrow U(3) 
         \rightarrow U(3)/\left(\strut U(1)^3\rtimes\Sigma_3\right).
\]
Just as before, we find that the map
$U(1)^3\rtimes\Sigma_3 \rightarrow U(3)$ induces a surjection on 
fundamental groups. Continuing the long exact sequence in homotopy gives us
\[
 \pi_{1} \left[ U(3)/\left(\strut U(1)^3\rtimes\Sigma_3\right) \right]
      \cong \pi_{0}\left[U(1)^3\rtimes\Sigma_3 \right] 
      \cong \Sigma_{3}. 
\]
A similar argument tells us that in the upper left corner
of~\eqref{eq: pushout diagram for L3}, we have
\[
\pi_{1} \left[\strut U(3)/I\right] \cong \pi_{0} I\cong \Sigma_{2}, 
\]
and the left vertical map on fundamental groups is the inclusion of 
$\Sigma_{2}\hookrightarrow\Sigma_{3}$ by $(1\ 2)\mapsto (1)\ (2\ 3)$.

The Seifert-Van Kampen theorem now tells us that the fundamental group 
of $\Lcal_{3}$, which is the homotopy pushout 
of~\eqref{eq: pushout diagram for L3}, is given by taking the free product of
the fundamental groups of the lower left and upper right, namely 
$\Sigma_{3}$ and the trivial group, and taking the quotient by the normal 
subgroup generated by the fundamental group of the upper left corner, 
which is $\Sigma_{2}$. But the smallest normal subgroup of $\Sigma_{3}$
containing $\Sigma_{2}$ is actually $\Sigma_{3}$ itself. 
We conclude that $\pi_1(\Lcal_3)$ is trivial.  
\end{proof}

\section{Group-theoretic results}\label{section:Group}

This section is devoted to establishing preliminary group-theoretic
results about $U(n)$ and related groups that we will need in later
sections.  More precisely, the goal of this section is to establish
criteria for $H\subseteq U(n)$ that allow us to find suitable elements
of $H$ that are central and of order $p$. These criteria will be
needed in order to apply Corollary~\ref{corollary:B4} and construct
contractions of fixed point spaces.

Let $PU(n)$ denote the projective unitary group, i.e., the quotient
of $U(n)$ by its center $S^1$. If $H$ is a subgroup of $U(n)$, we 
use $\Hbar$ to denote its image in $PU(n)$, so $\Hbar\cong H/(S^1\cap H)$.

We recall the following terminology from Section~3.1 of
\cite{GoodmanWallach}.  Let $v$ be a complex representation of a
group $G$; then $v$ splits into irreducibles as \[v\cong \bigoplus_k
v_k^{\oplus m_k}\] where for different values of $k$, the corresponding
irreducible representations $v_k$ are non-isomorphic. This decomposition is not
canonical; however, if we group all the isomorphic irreducibles
together into $w_k = v_k^{\oplus m_k}$, the decomposition
\[
v\cong \bigoplus_k w_k 
\]
is canonical. It is called the \emph{isotypic} decomposition of $v$,
and we call the subspaces $w_k$ the \emph{isotypic components}
of $v$. If $v$ has only
one isotypic component, we say it is an \emph{isotypic
representation}; otherwise, we say it is \emph{polytypic}.

For our situation, we will usually be considering a closed subgroup
$H\subseteq U(n)$ acting through the standard representation of $U(n)$
on $\complexes^{n}$. In this case we say that $H$ \emph{acts
isotypically/polytypically} or that $H$ \emph{is isotypic/polytypic} if
$\complexes^{n}$ is isotypic/polytypic as an $H$-representation. 

\begin{lemma} \label{lemma:B3} 
  Let $\integers/p$ be a subgroup of $PU(n)$ and let $J$ be its
  inverse image in $U(n)$. Then $J$ is polytypic.  In fact, a
  generator of $\integers/p\subseteq PU(n)$ can be lifted to an
  element of order $p$ in $U(n)$ and $J \cong S^1 \times \integers/p$.
\end{lemma}

\begin{proof}
  Let $A \in U(n)$ be such that its image $\Abar\in PU(n)$ generates
  the given $\integers/p$. Then $A^p$ is an element of the central $S^1
  \subseteq U(n)$ and is a diagonal matrix with all equal diagonal
  entries. Let $\alpha$ be some $p$th root of the diagonal entry of
  $A^p$. Define $B=\alpha^{-1} A$, so $B^p=\Id$ and the image of $B$ in 
  $PU(n)$ is $\Abar$. Then $\Abar\mapsto B$ determines a homomorphism
that splits the short exact sequence
\[
1\rightarrow S^1\rightarrow J\rightarrow \integers/p\rightarrow 1,
\]
and since $S^{1}$ is central, the map 
$S^1\times\integers/p\rightarrow J$ given by $(s,\Abar^b)\mapsto s B^b$ 
respects the multiplication in $J$ and is an isomorphism. 

It remains to show that $J$ is polytypic. Since $J$ is abelian, its
irreducible representations are all one-dimensional, so every element
of $J$ acts on every one-dimensional $J$-irreducible summand of
$\complexes^{n}$ by multiplication by a scalar. If $J$ were isotypic,
then an element of $J$ would have to act on every such summand 
by multiplication by the same scalar, i.e.,  $J$ would be contained
in $S^{1}$, which is false. Hence, $J$ is polytypic.
\end{proof}

The subgroups of $U(n)$ of greatest interest to us are the
$p$-toral subgroups.  We begin with a definition.

\begin{definition}\label{p-toral_def}
  A $p$-\emph{toral group} $H$ is an extension of a finite $p$-group
  by a torus.  In other words, there exists a short exact sequence
\[
1 \to T \to H \to P \to 1,
\] 
where $T$ is a torus and $P$ is a finite $p$-group. If $T\cong
(S^1)^{\times r}$, we call $r$ the \emph{rank} of $H$. The torus $T$
is the connected component of the identity of $H$, so we also
denote it by $H_0$.
\end{definition}

\begin{lemma} \label{lemma:HbarpToral} 
  Any quotient of a $p$-toral subgroup $H$ by a closed normal subgroup
  $K$ is $p$-toral.
\end{lemma}
In particular, if $H$ is a $p$-toral subgroup of $U(n)$, then its
image $\Hbar$ in $PU(n)$ is also $p$-toral.

\begin{proof}
  Since $H$ is $p$-toral, we have a short exact sequence 
\[1 \to T \to H \to P \to 1,
\]
where $T$ is a torus and $P$ is a $p$-group. We obtain a morphism of short
exact sequences
\[ 
\begin{CD}
1 @>>> T @>>> H @>>> P @>>> 1\\
@. @VVV @VVV @VVV \\
1 @>>> T / (K \cap T) @>>> H/K @>>>
  Q @>>> 1.
\end{CD}
\] 
The map $H\rightarrow H/K$ is surjective, as is $H/K\rightarrow Q$,
so $P\rightarrow Q$ is surjective, implying that $Q$ is a $p$-group. Further, 
$T/(K\cap T)$ is a compact connected abelian Lie group, and so must 
be a torus.
\end{proof}

For the remainder of this section, we focus on finding certain elements
of order~$p$ in $\Hbar$. 

\begin{definition}\label{H/p_def}
  For a $p$-toral group $H$, let $H/p$ denote the quotient of $H$
  by its normal subgroup $H^p[H,H]$, the
  normal subgroup generated by $p$-th powers and commutators.
 \end{definition}

 For a finite group $H$, the subgroup $H^p[H,H] $ may be familiar to
 the reader as the Frattini subgroup of $H$, and $H/p$ as the Frattini
 quotient of H.  A $p$-toral group $H$ is elementary abelian if and
 only if both $H^p$ and $[H,H]$ are the trivial subgroup of $H$, i.e.,
 if and only if $H^p[H,H]$, the subgroup generated by both, is
 trivial.  Further, $H/p$ is an elementary abelian $p$-group, and the
 map $H \to H/p$ is initial among homomorphisms from $H$ to finite
 elementary abelian $p$-groups; thus we refer to $H/p$ as the maximal
 elementary abelian quotient of $H$.

The following lemma is the main technical tool for the proof of 
Theorem~\ref{thm: main theorem} in Section~\ref{section:MainProof}.
Part of the statement is actually Lemma~6.5 of \cite{ADL2}, for which we
have given a streamlined proof here.

\begin{lemma} \label{lemma: element of order p} 
  Let $H$ be a $p$-toral subgroup of $U(n)$ with image $\Hbar$
  in~$PU(n)$. If $\Hbar$ is nontrivial and not elementary abelian,
  then there exists a central element in $\Hbar$ of order~$p$ that
  lies in the kernel of $\Hbar\rightarrow\Hbar/p$.
\end{lemma}

\begin{proof}
  By Lemma~\ref{lemma:HbarpToral}, $\Hbar$ is itself $p$-toral.  If
  $\Hbar$ is connected, then it is a (nontrivial) torus and has at
  least $p-1$ elements of order $p$. They are central in $\Hbar$
  because the torus is abelian, and map to the identity in $\Hbar/p$
  because a torus is $p$-divisible.

  Suppose that $\Hbar$ is not connected and has a nontrivial identity
  component $\Hbar_{0}$.  Let $\orderp{\Hbar_0}$ denote the group of
  elements of $\Hbar_0$ that have order $p$. The conjugation action of
  $\Hbar$ on itself preserves $\orderp{\Hbar_0}$, while $\Hbar_{0}$
  acts trivially on $\orderp{\Hbar_0}$ since $\Hbar_{0}$ is abelian,
  so we get an action of $\Hbar/\Hbar_0$ on $\orderp{\Hbar_0}$.  By
  assumption $\Hbar_0$ is nontrivial, so the set $\orderp{\Hbar_0}$
  has $p^{\rank(\Hbar_{0})} > 1$ elements.  The nontrivial $p$-group
  $\Hbar/\Hbar_0$ fixes the identity element in $\orderp{\Hbar_0}$, so
  by the orbit decomposition of $\orderp{\Hbar_{0}}$, there exist at
  least $p-1$ other elements of $\orderp{\Hbar_{0}}$ fixed by
  $\Hbar/\Hbar_0$. These elements are in $Z(\Hbar)$, and because
  $\Hbar_{0}$ is a torus and is $p$-divisible, we know they also map
  to the identity in $\Hbar/p$.

  Now suppose that $\Hbar$ is not connected and has trivial identity
  component, i.e., $\Hbar$ is a finite $p$-group.  Since $\Hbar$ is
  not elementary abelian, the subgroup $\Hbar^{p}[\Hbar,\Hbar]$ is a
  nontrivial subgroup of $\Hbar$, and it is also normal. But a
  nontrivial normal subgroup of a finite $p$-group must have
  nontrivial intersection with the center, and this gives the required
  element. 
\end{proof}

\section{Conditions for contractibility of 
the fixed point sets}
    \label{section:Conditions}

In this section, we turn to $\Lcal_{n}$ for a general $n$ and
establish preliminary criteria for the action of a subgroup $H\subseteq
U(n)$ on $\Lcal_n$ to have a contractible fixed point set
(Proposition~\ref{proposition:B1} and Theorem~\ref{thm:ContractJ}).
We begin with some terminology and notation.
We often think of an object $\lambda$ in $\Lcal_{n}$ as given by the
equivalence classes of an equivalence relation $\sim_{\lambda}$ on
$\complexes^n\setminus\{0\}$, where $x \sim_\lambda y$ if $x$ and $y$
are in the same subspace of the partition $\lambda$. We therefore
denote the set of subspaces of the partition $\lambda$ by
$\class(\lambda):= \{ v_1, \ldots, v_m \}$.

\begin{definition}
\hfill
\begin{enumerate}
\item A partition $\lambda$ is \emph{weakly fixed} by
  $H$, or \emph{weakly} $H$-\emph{fixed} if $x\sim_{\lambda}y$ implies
  $hx\sim_{\lambda}hy$ for all $h\in H$.  That is, the action of $H$
  on $\complexes^{n}$ stabilizes $\class(\lambda)$ as a set, although
  $H$ may permute the elements of $\class(\lambda)$ nontrivially.  We
  denote the full subcategory of $\Lcal_{n}$ whose objects are weakly
  $H$-fixed partitions of $\complexes^{n}$ by $\weakfixed{H}$.
\item A partition $\lambda$ is \emph{strongly fixed} by $H$, or
  \emph{strongly} $H$-\emph{fixed}, if $x\sim_{\lambda}hx$ for all
  $x\in\complexes^{n}\backslash\{0\}$ and all $h\in H$. That is, the
  action of $H$ on $\class(\lambda)$ is trivial.   We
  denote the full subcategory of $\Lcal_{n}$ whose objects are strongly
  $H$-fixed partitions of $\complexes^{n}$ by $\strongfixed{H}$.
\item A strongly $H$-fixed partition $\lambda$ is called
  $H$-\emph{isotypic} if each element of $\class(\lambda)$ is
  an isotypic representation of $H$. We denote the full subcategory of
  $H$-isotypic partitions by $\Iso{H}$.
\end{enumerate}
\end{definition}

We observe that 
\[
\Iso{H}\ \subseteq\ \strongfixed{H}\ \subseteq\ \weakfixed{H},
\]
and that in general the containments are strict. We are interested in 
conditions under which $\weakfixed{H}$ is contractible.  

\begin{proposition}\label{proposition:B1}
  Let $H\subseteq U(n)$ be connected and polytypic. Then $\weakfixed{H}$
  is contractible.
\end{proposition}

\begin{remark*}
  Proposition~\ref{proposition:B1} is not actually used directly in
  the proof of Theorem~\ref{thm: main theorem}. We nonetheless include
  it as it may be of independent interest, and it provides an uncomplicated
  exemplar of our methods.
\end{remark*}

\begin{proof}
The action of $H$ on the elements of $\class(\lambda)$ for $\lambda \in
\left(\Lcal_n\right)^{H}$ defines a continuous map 
\[
H \to \Sigma_{\class(\lambda)} 
\]
from $H$ to the symmetric group on $\class(\lambda)$. Since $H$ is
connected and $\Sigma_{\class(\lambda)} $ is discrete, this map is trivial.
Thus any weakly $H$-fixed partition must be strongly $H$-fixed, and 
it is sufficient to prove that $\strongfixed{H}$ is contractible. 

A strongly $H$-fixed partition is a decomposition of $\complexes^n$
into representations of $H$. Each of these representations can in turn
be decomposed into its isotypic components as an $H$-representation,
which defines a functor $\phi: \strongfixed{H}\to \Iso{H}$. There is also 
a natural transformation of the composite 
\[
\strongfixed{H}\ \xrightarrow{\ \phi\ }
          \ \Iso{H}\ \hookrightarrow
          \ \strongfixed{H}
\]
to the identity functor on $\strongfixed{H}$, while the other
composition is actually equal to the identity. It follows that
$\strongfixed{H}$ is homotopy equivalent to $\Iso{H}$.  However, $H$
is polytypic, so the decomposition of $\complexes^n$ into isotypic
components is a proper partition of $\complexes^{n}$. This partition
is a terminal object of $\Iso{H}$, whence $\Iso{H}$ is contractible.
\end{proof}

Recall that any partition $\lambda$ in $\Lcal_n$ corresponds to an
equivalence relation on points of $\complexes^n\setminus\{0\}$, and 
$x \sim_\lambda y$ if $x$ and $y$ are in the same subspace of the
partition $\lambda$.  We now define another, coarser equivalence relation which
incorporates the group action. 

\begin{definition}\label{orbit_partition_def}
  Let $J \subseteq U(n)$ be a subgroup, and let $\lambda$ be an
  element of $\Lcal_n$ corresponding to the relation~$\sim_{\lambda}$.
  We define a new equivalence relation $\sim_{\lambdaglom{J}}$
  by $x \sim_{\lambdaglom{J}} y$ if
  there exists $j\in J$ such that $x \sim_{\lambda} j y$, and we
  denote the associated partition by $\lambdaglom{J}$.
\end{definition}

In other words, the partition $\lambdaglom{J}$ is the minimal coarsening of
$\lambda$ that is strongly fixed by $J$. 

\begin{lemma} \label{lemma:normalinH} 
  Let $J$ be a normal subgroup of $H$ and let
  $\lambda~\in\weakfixed{H}$.  Assume that $\lambdaglom{J}$ is a
  proper partition of $\complexes^n$.  Then
  $\lambdaglom{J}\in~\weakfixed{H}$.
\end{lemma}

\begin{proof}
  We want to show that $\lambdaglom{J}$ is weakly $H$-fixed. Let $h \in H$
  and suppose that $x\sim_{\lambdaglom{J}} y$. By definition of
  $\lambdaglom{J}$, there exists some $j \in J$ such that $x
  \sim_\lambda jy$.  Since $\lambda$ is fixed by $H$, we have 
  $hx\sim_\lambda hjy = \left(hjh^{-1}\right)hy$. 
  Since $J$ is normal in $H$, the element
  $hjh^{-1}$ is in $J$, so $hx \sim_{\lambdaglom{J}} hy$.
\end{proof}

We now bring these results together to give conditions under which
$\weakfixed{H}$ is contractible. 

\begin{theorem}\label{thm:ContractJ}
  Let $H\subseteq U(n)$, and let $J$ be a normal subgroup of $H$ such
  that for every $\lambda\in\weakfixed{H}$, the partition
  $(\lambda/J)$ is a proper partition of $\complexes^n$. If $J$
  is polytypic, then
\[
\weakfixed{H} \simeq \weakfixed{H}\cap\Iso{J}
\]
and  $\weakfixed{H}$ is contractible.
\end{theorem}

\begin{proof}
  Under our assumptions and by Lemma~\ref{lemma:normalinH}, the
  assignment $\lambda \mapsto \lambdaglom{J} $ defines a functor
\[
\weakfixed{H} \to\weakfixed{H}\cap\strongfixed{J}.
\]
Since $\lambdaglom{J} $ is a coarsening of $\lambda$, there is a natural
transformation from the identity functor on $\weakfixed{H}$ to the composite
\[
\weakfixed{H} \to\weakfixed{H}\cap\strongfixed{J}
          \hookrightarrow  \weakfixed{H},
\]
showing that the induced composite map on classifying spaces is
homotopic to the identity. Since for $\lambda$ in
$\weakfixed{H}\cap\strongfixed{J}$, we know $\lambdaglom{J} =
\lambda$, it follows that the map on classifying spaces induced by the
functor $\lambda \mapsto \lambdaglom{J}$ is a deformation retraction,
giving a homotopy equivalence
\begin{equation}    \label{eq: second inclusion}
\weakfixed{H}\simeq\weakfixed{H}\cap\strongfixed{J} .
\end{equation}

Next, we show that
$\weakfixed{H}\cap\Iso{J}\hookrightarrow\weakfixed{H}\cap\strongfixed{J}$
is a homotopy equivalence.  Suppose that $\lambda \in
\weakfixed{H}\cap\strongfixed{J}$, and construct a new partition
$\tilde\lambda$ by taking each $v\in \class(\lambda)$, and refining it
into its $J$-isotypic components. We claim that $\tilde\lambda$ is
weakly fixed by $H$. Indeed, suppose the refinement $\tilde \lambda $
of $\lambda$ has $\class(\tilde\lambda)=\{ v_{i} \}_{i \in \Ical}$
with $\oplus_k v_{i_k} = v_i$ and all of the $v_{i_k}$ are irreducible
and isomorphic as representations of $J$.  Let
$h\in H$ and $j\in J$ be arbitrary, and let $x$ be an element of
$v_{i_k}$.  Then
\[ 
jhx = hh^{-1}jhx = hj^\prime x,
\]
for some $j^\prime $ in $J$. But $j^\prime x$ is an element of $v_{i_k}$.
Thus we conclude that $hv_{i_k}$ is a representation of $J$. 

Since $v_{i_k}$ is a $J$-representation, there is a corresponding map
$\rho_{i_k} \colon J \to \GL(v_{i_k})$ from $J$ to the linear automorphisms
of $v_{i_k}$. Define $\rho^h_{i_k} \colon J \to \GL(v_{i_k})$ by
$\rho^h_{i_k}(j) = \rho_{i_k} (h^{-1} j h)$. Since $$j hx = h (h^{-1}
j h)x= h \rho^h_{i_k}(j) x,$$ the map $x \mapsto hx$ defines an
isomorphism from the representation determined by $\rho^h_{i_k}$ to $h
v_{i_k}$.  Since $\rho_{i_k}$ is irreducible, so is $\rho^h_{i_k}$.
Thus $h v_{i_k}$ is irreducible. Moreover, if $v_{i_k}$ is isomorphic
to $v_{i_{k'}}$, then the representations corresponding to
$\rho^h_{i_k}$ and $\rho^h_{i_{k'}}$ are isomorphic as well, whence $h
v_{i_k}$ is isomorphic to $h v_{i_{k'}}$. It follows that $h v_i =
\oplus_k h v_{i_k}$ is $J$-isotypic.

Since $h$ fixes $\lambda$, we have that $h$ permutes the elements of
$\class(\lambda)$. Let $w \in \class(\lambda)$, and let $hw \in
\class(\lambda)$ denote its image. The previous paragraph shows that $h$
maps the $J$-isotypic components of $w$ to the $J$-isotypic components of
$hw$. It follows that $H$ fixes the refinement of $\lambda$ into its
$J$-isotypic components as claimed.

The original partition $\lambda$ is a coarsening of $\tilde\lambda$,
so there is a natural transformation from the composite
\[
\weakfixed{H}\cap\strongfixed{J} 
          \rightarrow \weakfixed{H}\cap\Iso{J}
          \hookrightarrow \weakfixed{H}\cap\strongfixed{J}
\]
to the identity functor on $\weakfixed{H}\cap\strongfixed{J}$. If
$\lambda$ is in $\weakfixed{H}\cap\Iso{J}$, then $\tilde\lambda=\lambda$,
so the map on classifying spaces induced by the functor 
$\lambda\mapsto\tilde\lambda$ is a deformation retraction, giving a 
homotopy equivalence
\begin{equation}  \label{eq: first inclusion}
\weakfixed{H}\cap\strongfixed{J}\simeq \weakfixed{H}\cap\Iso{J}.
\end{equation}
The homotopy equivalences in equations~\eqref{eq: second inclusion}
and~\eqref{eq: first inclusion}, taken together, establish the weak 
equivalence required in the theorem. 

It remains to show that $\weakfixed{H}\cap\Iso{J}$ is contractible.
Since we are assuming that $J$ is polytypic, the decomposition $\mu$
of $\complexes^n$ into the isotypic components of $J$ is a proper
partition and is terminal in $\Iso{J}$. The partition $\mu$ is also in
$\weakfixed{H}$, by the same argument we used above to prove
that the refinement of any $\lambda$ into its $J$-isotypic components
was weakly fixed by $H$.

Thus $\mu$ is a terminal element in
$\weakfixed{H}\cap\Iso{J}$
showing that this category has contractible classifying space, as desired.
\end{proof}

\section{Proof of the main theorem}
\label{section:MainProof}

In this section, our goal is to use the results of the previous sections
to prove Theorem~\ref{thm: main theorem}. Let $H$ be a $p$-toral subgroup of 
$U(n)$ and let $\Hbar$ be its image in $PU(n)$. Our plan is to pick out
an appropriate element of order $p$ in $\Hbar$, lift it to $U(n)$ using 
Lemma~\ref{lemma:B3}, and then use Theorem~\ref{thm:ContractJ}.

\begin{lemma} \label{lemma:B2part:nontransitive}
Let $H\subseteq U(n)$ be a subgroup with $\Hbar$ its
  projection to $PU(n)$, and suppose there exists a subgroup 
$V\cong \integers/p$ of $\Hbar$ such that
\[
V \subseteq \ker\left(\Hbar \rightarrow \Hbar/p \right).
\]
Then $V$ does not act transitively on $\class(\lambda)$
for any weakly $H$-fixed partition $\lambda$ of $\complexes^n$.
\end{lemma}

\begin{proof}
  Let $\lambda$ be an object of $\weakfixed{H}$. Since $V \subseteq
  \Hbar$, it follows that $V$ permutes the elements of $\class(\lambda)$.
  We claim this permutation is not transitive. Assume the contrary. As
  $\lambda$ is a proper partition, it follows $|\class(\lambda)|>1$, so
  $|\class(\lambda)|=p$. We choose a bijection
  $\class(\lambda)\cong \{1,2,\dots,p \}$ such that the image of $V$
  under the resulting map $\Hbar\to \Sigma_p$ is generated by the
  $p$-cycle $(1,2,\dots,p)$. 
  
  As $\Hbar$ is $p$-toral, its image under the map $\Hbar \to
  \Sigma_p$ is a $p$-group. By the assumption that $V$ permutes the
  elements of $\class(\lambda)$ transitively, the image of $\Hbar$
  must contain the subgroup $\integers/p\subseteq\Sigma_{p}$ generated
  by $(1,2,\dots,p)$. This subgroup is a maximal $p$-subgroup of
  $\Sigma_{p}$, which forces the image of $\Hbar$ to be contained in
  $\integers/p$.  Therefore $\Hbar \to \Sigma_p$ factors through
  $\Hbar \to \Hbar/p$.  Because the restriction to $V $ is nontrivial,
  the existence of this factorization contradicts the assumption that
  $V\subseteq \ker\left(\Hbar \rightarrow \Hbar/p \right)$. We
  conclude that $V$ cannot act transitively on $\class(\lambda)$, thus
  proving the lemma.
\end{proof}

\begin{corollary} \label{corollary:B4} 
  Let $H$, $\Hbar$, and $V$ be as in Lemma~\ref{lemma:B2part:nontransitive}, and suppose further that $V$ is normal in $\Hbar$. 
  Then $\weakfixed{H}$ is contractible.
\end{corollary} 

\begin{proof}
	Let $J$ be the inverse image of $V$ in $H$; by Lemma~\ref{lemma:B3}, we know that $J$ is
  polytypic.  By Lemma~\ref{lemma:B2part:nontransitive}, $V$ does not
  act transitively on $\class(\lambda)$ for any weakly $H$-fixed
  partition~$\lambda$, and since the action of $J$ on $\Lcal_{n}$
  factors through $V$, we conclude that $J$ has the same property. 
  Since $V$ is normal in $\Hbar$, its inverse image $J$ is normal in
  $H$, so it
  follows that $J$ and $H$ satisfy the conditions of
  Theorem~\ref{thm:ContractJ}. Hence $\weakfixed{H}$ is contractible.
\end{proof}

We now have all the ingredients needed for the proof of our main theorem. 

\begin{maintheorem}
\maintheoremtext
\end{maintheorem}

\begin{proof}
  Let $H$ be a $p$-toral subgroup of $U(n)$.  Assume that $\Hbar$ is
  not a finite, elementary abelian $p$-group; then we want to prove
  that $\weakfixed{H}$ is contractible. Since $\Hbar$ is necessarily
  nontrivial, by Lemma~\ref{lemma: element of order p} there exists
\[
  V\cong \integers/p 
   \subseteq Z(\Hbar) \cap \ker\left(\Hbar \rightarrow \Hbar/p \right).
\]
 Thus we have
  $V\triangleleft\Hbar$ that satisfies the hypotheses of
  Lemma~\ref{lemma:B2part:nontransitive} and Corollary~\ref{corollary:B4}, and the theorem follows.
\end{proof}

\section{Examples}
\label{section: examples of fixed points} 

In this section we consider $p=2$ and compute two examples of fixed
point sets $\weakfixed{H}$ where $H$ is a projective elementary
abelian $p$-group. We first consider an example
$H\cong\integers/2\subseteq U(2)$ acting on $\Lcal_{2}$, and the fixed
points $\weakfixed{H}$ turn out not to be contractible
(Proposition~\ref{prop: non-contractible}).  This example shows that
Theorem~\ref{thm: main theorem} is group-theoretically sharp: there
exist projective elementary abelian $p$-groups with noncontractible
fixed point sets.  The second example is $H'\cong\integers/2\subseteq
U(3)$ acting on $\Lcal_{3}$. In this case the fixed point set turns
out to be contractible 
(Proposition~\ref{proposition: L3 fixed contractible}), illustrating
that not all projective elementary abelian $p$-subgroups have
noncontractible fixed point sets.  In~\cite{Banff-follow-on}, we will
establish further restrictions of a representation-theoretic nature
that narrow down the $p$-toral subgroups of $U(n)$ that can have
noncontractible fixed point sets on $\Lcal_{n}$.\\

First we compute fixed points on $\Lcal_{2}$. 
Let $H\cong\integers/2\subseteq U(2)$ be the subgroup 
generated by $\tau\in U(2)$ represented by the matrix
$\begin{bmatrix}
  0 & 1\\
  1 & 0
\end{bmatrix}$.\\

\begin{proposition} \label{prop: non-contractible} 
  The fixed point space $\left(\Lcal_2\right)^{H}$ is homeomorphic
  to the space $S^1 \sqcup \ast$.
\end{proposition}

To prove Proposition~\ref{prop: non-contractible}, we set up a little
notation.  Let $L_{z}$ denote the line in $\complexes^{2}$ spanned by
$(1,z)$, and let $L_{\infty}$ denote the line spanned by $(0,1)$.  We
saw in the proof of Proposition~\ref{prop: RP2} that the set of
objects in $\Lcal_{2}$ consists of pairs $\{L_{z}, L_{-1/\zbar}\}$
where $z\in\complexes\backslash\{0\}$, together with one extra point
$\{L_{0}, L_{\infty}\}$. Since $\tau\in U(2)$ exchanges the standard
basis vectors of $\complexes^{2}$, if $z\in\complexes\backslash\{0\}$
then $\tau\left(L_{z}\right)=L_{1/z}$, and $\tau$ exchanges $L_{0}$
and $L_{\infty}$.

\begin{lemma}  \label{lemma: fixed points as set}
As a set, the fixed points of the action of $\tau$ on $\Lcal_{2}$ consist of
the point $\{L_{1}, L_{-1}\}$, the point $\{L_{0}, L_{\infty}\}$,
and the set of points
$\{L_{ir}, L_{-i/r}\}$ where $r\in\reals\backslash\{0\}$.
\end{lemma}

\begin{proof}
Direct computation establishes that the points of $\Lcal_{2}$
in the statement of the lemma are in fact fixed by $\tau$.

Points in $\Lcal_{2}$ besides $\{L_{0}, L_{\infty}\}$ have the form
$\{L_{z}, L_{-1/\zbar}\}$ where $z\in\complexes\backslash\{0\}$. If
such a point is fixed by $\tau$, then either each line in the pair is
fixed by $\tau$ (the partition is strongly fixed), or else the lines
are interchanged by $\tau$ (the partition is only weakly fixed). In the first
case, since $\tau\left(L_{z}\right)=L_{1/z}$, we must have $z=1/z$, so
$z=\pm 1$, corresponding to the point $\{L_{1}, L_{-1}\}$. In the
second case, we must have $1/z=-1/\zbar$, meaning $\zbar=-z$, so $z$
is purely imaginary, say $z=ir$ for $r\in\reals\backslash\{0\}$.
Thus $\{L_{z}, L_{-1/\zbar}\}$ has the form $\{L_{ir}, L_{-i/r}\}$.
\end{proof}

\begin{proof}[Proof of Proposition~\ref{prop: non-contractible}]
 To determine the fixed point set of the action of
  $\tau$ on $\Lcal_{2}$ as a topological space, and not just as a set
  (as in Lemma~\ref{lemma: fixed points as set}), we recall from the
  proof of Proposition~\ref{prop: RP2} that $\Lcal_{2}$ can be
  identified as the quotient space of the disk $\|z\|\leq 1$ in
  $\complexes^{1}$ by the antipodal action on the boundary circle
  $\|z\|=1$.  According to Lemma~\ref{lemma: fixed points as set}, the
  fixed points of $\tau$ correspond to the points of the unit disk
  that lie on the imaginary axis, $\{ir \mid r\in
  [-1,1]\subseteq\reals\}$, together with the real points $1$ and
  $-1$. The points $1$ and $-1$ are identified by passing to
  $\Lcal_{2}$, as are the points $i$ and $-i$, which gives
  $S^{1}\sqcup \ast$ as the fixed point space of $\tau$ acting on
  $\Lcal_{2}$.
\end{proof}
 
Proposition~\ref{prop: non-contractible} gives an example of a
$\integers/2\subseteq U(2)$ that acts with noncontractible fixed point
set on $\Lcal_{2}$.  However, when we take the same subgroup and embed
it in $U(3)$, we get a different result. Let
$H'\cong\integers/2\subseteq U(3)$ be the subgroup generated by 
\[\tau'=
\begin{bmatrix} 0 & 1 & 0\\
    1 & 0 & 0\\
    0 & 0 & 1
\end{bmatrix}\in U(3). 
\]
\\

\begin{proposition}   \label{proposition: L3 fixed contractible}
The fixed point space $\weakfixedspecific{H'}{3}$
is contractible.
\end{proposition}

\begin{proof}
  We use essentially the proof of Theorem~\ref{thm:ContractJ},
  slightly modified. 

  First, we assert that if $\lambda$ is weakly fixed by $H'$, then the
  partition $\lambdaglom{H'}$ described in
  Definition~\ref{orbit_partition_def} is a proper partition of
  $\complexes^{3}$. There are two cases: either $\class(\lambda)$ consists
  of three lines, in which case $H'\cong\integers/2$ cannot act on it
  transitively, or else $\lambda$ has one line and one two-dimensional
  subspace, in which case $H'$ cannot interchange them because they
  have different dimensions. As a result, the inclusion
\[
\strongfixedspecific{H'}{3}\hookrightarrow \weakfixedspecific{H'}{3}
\]
has a functorial retraction $\lambda\mapsto\lambdaglom{H'}$, and
induces a homotopy equivalence.

Exactly as in the proof of Theorem~\ref{thm:ContractJ}, we note
that 
\[
\Isospecific{H'}{3}\hookrightarrow \strongfixedspecific{H'}{3}
\]
also has the functorial retraction that takes an object $\lambda$ in 
$\strongfixedspecific{H'}{3}$ to its refinement into $H'$-isotypic
classes. Therefore each of the inclusions 
\[
\Isospecific{H'}{3}\hookrightarrow \strongfixedspecific{H'}{3}
                   \hookrightarrow \weakfixedspecific{H'}{3}
\]
induces a homotopy equivalence. 

To finish the proof we need to identify the homotopy type 
of~$\Isospecific{H'}{3}$. The action of $\tau'$ on
$\complexes^{3}$ has eigenvalues $1$ (with multiplicity $2$) and $-1$.
Therefore $H'$ acts polytypically on $\complexes^{3}$, and 
$\Isospecific{H'}{3}$ has a final object consisting of the 
canonical decomposition of $\complexes^{3}$ into isotypic representations
of $H'$: 
\[
\mu=\{\{(u,u,v)|u,v\in\complexes\},\{(u,-u,0)|u\in\complexes\}\}. 
\]

We conclude that 
$\weakfixedspecific{H'}{3}\simeq \strongfixedspecific{H'}{3}
                          \simeq \Isospecific{H'}{3}
                          \simeq \ast$. 
\end{proof}

 \bibliographystyle{amsalpha}
 	\bibliography{unitary}

\end{document}